\documentclass[11pt]{amsart}
\usepackage{mathrsfs,amssymb}
\usepackage{times,amsthm,amscd,pstricks, graphicx, pstricks-add}
\usepackage[all]{xy}

\theoremstyle{plain}
\newtheorem{theorem}{Theorem}[section]
\newtheorem{lemma}[theorem]{Lemma}
\newtheorem{proposition}[theorem]{Proposition}

\theoremstyle{definition}
\newtheorem{example}[theorem]{Example}
\newtheorem{definition}[theorem]{Definition}
\theoremstyle{remark}
\newtheorem{remark}[theorem]{Remark}

\renewcommand{\hom}{\operatorname{Hom}}
\newcommand{\End}{\operatorname{End}}

\newcommand{\id}{\operatorname{id}}

\newcommand{\cobar}{\operatorname{\Omega}}

\newcommand{\cofrob}{\mathrm{co}\mathbf{Frob}}

\newcommand{\sS}{\mathbb{S}}
\newcommand{\fieldk}{\mathbf{k}}

\begin{document}

\title{Algebras over $\cobar(\cofrob)$}

\author{Gabriel C. Drummond-Cole}
\address{CUNY Graduate Center, City University of New York, 
    365 Fifth Avenue, 
New York, NY 10016-4309}
\email{gdrummond-cole@gc.cuny.edu}
\author{John Terilla}
\address{Queens College and CUNY Graduate Center, City University of New York, 
    Flushing, NY 11367}
\email{jterilla@qc.cuny.edu}
\author{Thomas Tradler}
\address{NYC College of Technology and CUNY Graduate Center, City University of New York, Brooklyn, NY 11201}
\email{ttradler@citytech.cuny.edu}

\begin{abstract}We show that a square zero, degree one element in
  $W(V)$, the Weyl algebra on a vector space $V$, is equivalent to
  providing $V$ with the structure of an algebra over the properad
  $\cobar(\cofrob)$, the properad arising from the cobar construction
  applied to the cofrobenius coproperad.
\end{abstract}

\maketitle 
\tableofcontents

\section{Introduction}
The main result in this paper is Theorem \ref{T:main} which asserts
that two algebraic structures defined on a vector space $V$ are the
same.  One structure is defined by a square zero, degree one element
in $W(V)$, the Weyl algebra on $V$.  In the next few paragraphs, we
give an brief summary of the Weyl algebra, and we give the precise
definitions and formulae with signs in Sections \ref{S:SV-TV-review}
and Section \ref{S:Weyl}. The other structure is an algebra over the
properad $\cobar(\cofrob)$.  Roughly speaking, properads and
coproperads are constructs that model composable and decomposable
operations to and from the tensor powers of a vector space.  The main
reference is \cite{V} and we give a review of properads and
coproperads, mostly to fix notation and conventions, in Section
\ref{S:properads}.  In the same way that operads govern algebras with
many-to-one operations, properads and coproperads govern algebras with
many-to-many operations, such as Lie bialgebras, and are built to
accomodate the ``higher genus'' phenomena which may arise from
composing multiple outputs with multiple inputs, such as the
involutive relation possessed by certain Lie bialgebras.  Of
particular interest here are the relations organized by genus arising
from the Frobenius compatibility in Frobenius algebras.  There is a
simple coproperad, call it $\cofrob$, determined by the Frobenius
relations, and we review it in Section \ref{SS:cofrob}.  The symbol
$\cobar$ in the expression ``$\cobar(\cofrob)$'' denotes a general
construction called the cobar construction, which assigns a properad
to certain coproperads.  We review the cobar construction in Section
\ref{SS:cobar}.

Now, we present an overview of how we define the Weyl algebra.  Fix a
ground field $\fieldk$ of characteristic zero.  Let $V$ be a graded
vector space over $\fieldk$, let $S^kV$ be the $k$-th symmetric
product of $V$, let $SV=\oplus_{k=0}^\infty S^kV$ be the symmetric
algebra of $V$, and let $\widehat{SV}=\prod_{k=0}^\infty S^k V$.
Consider the $k[[\hbar]]$ module $\hom(SV,\widehat{SV})[[\hbar]]$.
There exists a star product
$$\star:\hom(SV,\widehat{SV})[[\hbar]]\otimes_{k[[\hbar]]}\hom(SV,\widehat{SV})[[\hbar]]\to
\hom(SV,\widehat{SV})[[\hbar]]$$ which is an associative,
noncommutative, degree zero map of $k[[\hbar]]$ modules.  We let
$W(V):=\left(\hom(SV,\widehat{SV})[[\hbar]],\star\right)$, and call it
the Weyl algebra\footnote{We use the name ``Weyl'' since Hermann Weyl
  used a prototype of this algebra in his work in quantum mechanics
  (see Chapter 2, section 11 of \cite{W}), although the term ``star
  product'' was introduced later \cite{FLS}. The reader interested in
  the rich history of Weyl algebras and star products may wish to
  consult \cite{DS}, and the references therein.} of $V$.  We define
the star product in a coordinate free way which is also natural from
the point of view of maps between tensor powers of vector spaces, but
we pay for our choice to be choice-free with combinatorial factors
(banished to Appendix \ref{A:combinatorics}) which are used to align
our definition with the familiar coordinate-dependent presentation in
common use since at least 1928 (\cite{W}).  Lemma \ref{L:Weyl-product}
states that our definition is equivalent to the traditional one.

The star product is determined by its values on pairs $f,g\in
\hom(SV,\widehat{SV})$ and decomposes in powers of $\hbar$ by
$$f\star g = f \circ_1 g + \hbar f\circ_1 g + \hbar^2 f\circ_2
g+\cdots$$ We are interested in degree negative one elements $H\in
W(V)$ satisfying $H\star H=0$.  Any element $H\in W(V)$ in the Weyl
algebra comprises a collection of operators
$\left(\sigma_{(g)}\right)_i^j:S^iV \to S^jV$, $g,i,j\geq 0$:
decompose $H$ into pieces $H=\sigma_{(0)}+\hbar
\sigma_{(1)}+\hbar^2\sigma_{(2)}+\cdots$ where each
$\sigma_{(g)}:SV\to \widehat{SV}$ and decompose each map
$\sigma_{(g)}$ into operators $\left(\sigma_{(g)}\right)_i^j:S^iV \to
S^jV$.  The condition that $H\star H=0$ summarizes an infinite
collection of relations among the maps
$\left(\sigma_{(g)}\right)_i^j$.  In this paper, we make a technical
assumption on $H$ that $\left(\sigma_{(g)}\right)_i^j=0$ if either $i$
or $j$ or both are zero in order to avoid certain difficulties when we
compare $H$ with a properadic structure.

It is no surprise that degree one, square zero elements of the Weyl
algebra make up the data of an algebra over a properad.  The work
about which we are reporting consists mostly of identifying the
properad precisely, and working through the signs and combinatorial
factors.  A motivation for the work is that elements of square zero in
the Weyl algebra appear in a number of settings---they figure
prominently in a mathematical interpretation of quantum field theory
that grew from the BV-quantization scheme \cite{PTT}; and the deep
compactification, gluing, and analysis theorems and conjectures in
symplectic geometry can be summarized as a square zero, degree one
element $H$ in the Weyl algebra of a vector space defined by the Reeb
orbits of a contact manifold \cite{EGH}.

In the last section of the paper, Section \ref{S:homology}, we verify
that the homology of an algebra over $\cobar(\cofrob)$ is a
(commutative) Lie bialgebra satisfying the involutive relation.  We
conjecture that the $\cobar(\cofrob)$ properad gives a resolution of
the Lie bialgebra properad, but at present we do not have a proof
(computer computations show that if $\cobar(\cofrob)$ is not a
resolution of involutive biLie, one must look at rather high Euler
characteristic to find a nonzero homology class).  One implication of
Section \ref{S:homology} is that from a degree one element $H\in W(V)$
with $H\star H=0$, one obtains an associated homology theory which has
the structure of a Lie bialgebra.  In the case of the $H$ from
symplectic field theory, the involutive Lie bialgebra in homology is
known to contact geometers \cite{CL}, see also \cite{STT}.

The authors would like to thank the referee for the helpful
suggestions that sharpened the exposition.

\section{Review of symmetric and tensor
  algebras}\label{S:SV-TV-review}
In this section, we fix some basic notation that is used throughout
this paper, and define partial gluing operations $\circ_k$ which are
used in our definition of the Weyl algebra.  Some of the complexity in
the next couple of sections arises from passing between viewing the
symmetric algebra as a \emph{quotient} of the tensor algebra and
viewing the symmetric algebra as a \emph{subalgebra} of the tensor
algebra.  When we view $SV$ as the free commutative algebra on $V$, it
is naturally obtained as a quotient of the free algebra $TV$.  In
characteristic zero, $SV$ and $TV$ are also constructions of the free
cocommutative coalgebra and the free coalgebra on $V$, in which case
$SV$ naturally embeds in $TV$.  The specifics follow, but the reader
may wish to skim over the signs and combinatorics and jump to Figure
\ref{fig1} which gives a picture of the partial gluing operations
$\circ_k$ used to define the Weyl algebra in Definition \ref{D:Weyl}.

For any element $v$ in the graded vector space $V$, let $|v|$ denote
the degree of $v$.  Let $T^nV$ and $TV$ denote the corresponding
tensor power and tensor algebra of $V$.  The tensor product is denoted
by $\otimes$ and the symmetric product by $\odot$.  The element
$v_1\otimes\cdots\otimes v_k\in T^kV$ will be denoted by $\bar{v}$.

If $\sigma$ is in $\sS^k$, the symmetric group on $k$ letters, then
let $\epsilon(\sigma,\bar{v})$ be the Koszul sign, i.e. the sign of
the permutation induced by $\sigma$ on the odd entries of $\bar{v}$.
Then there is a left $\sS^k$-action on $T^kV$ defined as the linear
extension of $\sigma(v_1\otimes\cdots\otimes
v_k)=\epsilon(\sigma,\bar{v})v_{\sigma^{-1}(1)}\otimes\cdots\otimes
v_{\sigma^{-1}(k)}$.  The image of $\bar {v}$ under the action of
$\sigma$ will be denoted $\sigma \bar v$.
The sign $\epsilon(\phi, \bar{v})$ where $\phi$ is the permutation
which reverses the order of $\bar{v}$ appears occasionally; denote it
by $||\bar{v}||$.  This sign depends only on the number of odd entries
of $\bar{v}$; if this number is $0$ or $1$ mod $4$, the sign is
positive; if it is $2$ or $3$ the sign is negative.  It also satisfies
$(||\bar{u}||)(||\bar{v}||)(||\bar{u}\otimes \bar{v}||) = (-1)^{| \bar
  u| | \bar v|}$.

A $k,\ell$ shuffle $\sigma$ is an element of $\sS_k$ such that
$\sigma(1)<\cdots<\sigma(\ell)$ and $\sigma(\ell+1)<\cdots<\sigma(k)$.
Let the set of shuffles be $\sS_{k,\ell}$.  An unshuffle is an element
of $\sS_{k,\ell}^{-1}$.  If $\tau$ is an unshuffle, then
$\bar{v}_\ell$ and $\bar{v}_{k-\ell}$ should be taken to mean the
first $\ell$ and the last $k-\ell$ factors of $\tau(\bar{v})$,
respectively.  The suppression of $\tau$ should not cause much
confusion.  Any permutation in $\sS_k$ can be factored uniquely as the
composition of a $k,\ell$ shuffle with a permutation from
$\sS_\ell\otimes \sS_{k-\ell}$ and uniquely as the composition of a
permutation from $\sS_\ell\otimes \sS_{k-\ell}$ with a $k, \ell$
unshuffle.

The space $S^kV$ is defined as the quotient of $T^kV$ by the subspace
spanned by $v-\sigma v$, where $v$ ranges over $T^kV$ and $\sigma$
ranges over $\sS_k$. Let the symmetric class of an element $v$ of
$T^kV$ be denoted $[v]$; let $s^k:T^kV\to S^kV$ denote the projection
$v\mapsto [v]$.  Also, oweing to the fact that in characteristic zero,
the algebras $TV$ and $SV$ are constructions of the free coalgebra and
free commutative coalgebra on $V$, one can embed $S^kV$ in $T^kV$ via
the symmetrization map $\iota^k:S^kV\to T^kV$ defined by
$$
\iota^k [v]=\frac{1}{k!}\sum_{\sigma\in \sS^k}\sigma v.
$$
The superscript $k$ in $\iota^k$ and $s^k$ will usually be suppressed.
Observe that $s \iota [v] = \frac{1}{k!}\sum [\sigma v]=[v]$ and
$\iota s (v)=\frac{1}{k!}\sum \sigma v$.  If $v\in T^kV$ satisfies
$\sigma v=v$ for all $\sigma \in \sS_k$, then $\iota s v=v.$

Using the above notation, we now define the partial gluing maps, both
on the tensor algebra $TV$ and the symmetric algebra $SV$.
\begin{definition}
  For $i,j,k,m,n\ge 0$, there is a \emph{partial gluing map}
  $\circ_k:\hom(T^mV,T^nV)\otimes \hom(T^iV,T^jV)\to
  \hom(T^{m+i-k}V,T^{n+j-k}V)$ given by, $$\varphi\circ_k \psi =
  (\varphi\otimes \id^{\otimes j-k})\circ (\id^{\otimes m-k}\otimes
  \psi)$$ and is depicted in the following figure.

  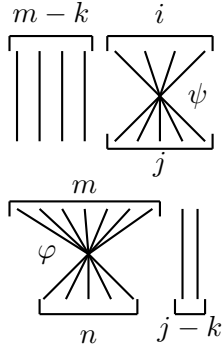
\begin{figure}[h]
    \begin{pspicture}(0,0)(5,4.8)
      \psline(1,4)(1,2.8) \psline(1.3,4)(1.3,2.8)
      \psline(1.6,4)(1.6,2.8) \psline(1.9,4)(1.9,2.8)
      \psline(2.3,4)(2.9,3.4) \psline(2.7,4)(2.9,3.4)
      \psline(3.1,4)(2.9,3.4) \psline(3.5,4)(2.9,3.4)
      \psline(2.9,3.4)(2.9,2.8) \psline(2.9,3.4)(2.6,2.8)
      \psline(2.9,3.4)(2.3,2.8) \psline(2.9,3.4)(3.2,2.8)
      \psline(2.9,3.4)(3.5,2.8) \rput(3.4,3.4){$\psi$}
      \psline(.9,4.2)(.9,4) \psline(.9,4.2)(2,4.2) \psline(2,4.2)(2,4)
      \rput(1.45,4.5){$m-k$} \psline(2.2,4.2)(3.6,4.2)
      \psline(2.2,4.2)(2.2,4) \psline(3.6,4.2)(3.6,4)
      \rput(2.9,4.5){$i$} \psline(2.2,2.7)(3.6,2.7)
      \psline(2.2,2.7)(2.2,2.9) \psline(3.6,2.7)(3.6,2.9)
      \rput(2.9,2.5){$j$} \psline(.9,2)(2.9,2) \psline(.9,2)(.9,1.8)
      \psline(2.9,2)(2.9,1.8) \rput(1.9,2.2){$m$}
      \psline(3.2,1.9)(3.2,.7) \psline(3.4,1.9)(3.4,.7)
      \psline(3.1,.7)(3.1,.5) \psline(3.5,.7)(3.5,.5)
      \psline(3.1,.5)(3.5,.5) \rput(3.3,.3){$j-k$}
      \psline(1,1.9)(1.95,1.3) \psline(1.3,1.9)(1.95,1.3)
      \psline(1.6,1.9)(1.95,1.3) \psline(1.9,1.9)(1.95,1.3)
      \psline(2.2,1.9)(1.95,1.3) \psline(2.5,1.9)(1.95,1.3)
      \psline(2.8,1.9)(1.95,1.3) \psline(1.95,1.3)(1.95,.7)
      \psline(1.95,1.3)(1.65,.7) \psline(1.95,1.3)(1.35,.7)
      \psline(1.95,1.3)(2.25,.7) \psline(1.95,1.3)(2.55,.7)
      \psline(1.3,.7)(1.3,.5) \psline(2.6,.7)(2.6,.5)
      \psline(1.3,.5)(2.6,.5) \rput(1.95,.2){$n$}
      \rput(1.4,1.3){$\varphi$}
    \end{pspicture}
    \caption{Depiction of the partial gluing operation
      $\circ_k$ \label{fig1} where the first $k$ outputs of $\psi$ are
      composed with the last $k$ inputs of $\varphi$}
  \end{figure}
  There is also an induced \emph{partial composition map}, by abuse of
  notation also denoted $\circ_k:\hom(S^m V,S^nV)\otimes
  \hom(S^iV,S^jV)\to \hom(S^{m+i-k}V,S^{n+j-k}V)$ defined by
$$g\circ_k f= 
\binom{m+i-k}{i}\binom{j}{k}s((\iota g s)\circ_k (\iota f s))\iota
$$
The reason for the choice of combinatorial factors in this definition
is due to the property exhibited in Proposition \ref{L:compare-circ_k}
of appendix \ref{A:combinatorics}.
\end{definition}
\begin{remark}
  By convention, $\id^{\otimes \ell}=0$ when $\ell<0$, so that the
  partial gluing map $\circ_k$ is zero when $k>m$ or $k>j$.
\end{remark}
\begin{remark}
  Note that the definition of $g\circ_k f$ for maps $g:S^mV \to S^nV$
  and $f:S^iV\to S^jV$ extends to all of $\hom(SV,\widehat{SV})$ since
  there are only finitely many contributions to $\circ_k$.
\end{remark}

\section{The Weyl algebra}\label{S:Weyl}
In this section we define the Weyl algebra of a vector space $V$ over
a field $\fieldk$.  The coordinate free definition that we give will
be a $\fieldk[[\hbar]]$ algebra on $\hom(SV,\widehat{SV})[[\hbar]]$.

\begin{definition}\label{D:Weyl}
  We define the \emph{Weyl algebra of $V$} to be the
  $\fieldk[[\hbar]]$ algebra $(W(V),\star)$ where
  $$W(V)=\hom(SV,\widehat{SV})[[\hbar]]$$ and
  $$\star:W(V)\otimes_{k[[\hbar]]}W(V)\to W(V)$$ is defined for
  $f,g\in \hom(SV,\widehat{SV})$ by
$$g\star f= g\circ_0 f+(g\circ_1 f)\hbar + (g\circ_2 f)\hbar^2
+\cdots$$
\end{definition}

One frequently encounters this $\star$ product for a finite
dimensional vector space $V$ and ``in coordinates.''  Traditionally,
elements of $V$ are denoted by $q$'s and elements of its dual space
$V^*=\hom(V,\fieldk)$ are denoted by $p$'s (position and momentum).
If $\{q_\ell\}$ is a homogeneous basis for $V$ with dual basis
$\{p^\ell\}$ of $V^*$, elements of $\hom(SV,\widehat{SV})$ are power
series in the $p^\ell$'s and the $q_\ell$'s.  Maps $f:S^iV \to S^jV$
and $g:S^mV\to S^nV$ are expressed in a standard form with all the
$p$'s on the right
\begin{equation}\label{typical_f}
  f=\sum f_{\bar{i}}^{\bar{j}}q_{\bar{j}}
  p^{\bar{i}}\text{ and } g=\sum
  g_{\bar{m}}^{\bar{n}}q_{\bar{n}}p^{\bar{m}}.
\end{equation}
Here, $\bar{j}$, $\bar{i}$, $\bar{m}$, and $\bar{n}$ are
multi-indices, and we will use the same notational conventions that we
do for tensor products: $\bar{j}_k$ and $\bar{j}_{j-k}$ denote the
multi-indices consisting of the first $k$ indices of $\bar{j}$ and the
last $j-k$ indices of $\bar{j}$ respectively, and $\phi(\bar{j})$ will
be the reverse of $\bar{j}$.  We distinguish vectors in the tensor
algebra from the symmetric algebra by using a tensor symbol in the
subscript: if, for example, $\bar{j}=(5,2,8)$, then
$q_{\bar{j}}=q_5\odot q_2\odot q_8\in SV $ and $q_{\otimes
  \bar{j}}=q_5 \otimes q_2 \otimes q_8\in TV.$ The symbol
$\delta_{\bar{i}}^{\bar{j}}$ is zero unless $\bar{i}=\bar{j}$, in
which case $\delta_{\bar{i}}^{\bar{j}}=1$.

The function $f$ in Equation \eqref{typical_f} for instance maps
$q_{\bar{k}} =[q_{\otimes \bar{k}}]$ to
$$\sum_{\sigma\in \sS_k}\epsilon(\sigma,\bar{k}) \delta^{\phi\bar{i}}_{\sigma\bar{k}}f^{\bar{j}}_{\bar{i}}[q_{\otimes \bar{j}}]$$
Note that the $p$'s act in ``reverse'' order, which gives the standard
signs when translated to a tensor algebra context for a graded vector
space.  In other words,
$$f=i!f_{\bar{i}}^{\bar{j}}[q_{\otimes \bar{j}}]p^{\otimes \bar{i}}\iota$$
\begin{lemma}\label{L:Weyl-product}
  The product $g\star f$ is the free product on the formal power
  series in the variables $\{q_\ell,p^\ell\}$ subject to the relations
  \begin{gather*} [p^{\ell},q_{\ell'}]:=p^\ell
    q_{\ell'}-(-1)^{|p^\ell||q_{\ell'}|}q_{\ell'}p^\ell=\hbar
    \delta^{\ell}_{\ell'}\\
    [p^{\ell},p^{\ell'}]=[q_\ell,q_{\ell'}]=0.
  \end{gather*}
\end{lemma}
\begin{remark}
  The significance of the $p$-$q$ description of the Weyl algebra is
  that the symplectic nature of the situation is illuminated:
  $\hom(SV,\widehat{SV})$ can be viewed as (a completion of) the
  polynomial functions on the symplectic vector space $V\oplus V^*$.
  As usual, the set of such functions forms a Poisson algebra.  In the
  notation of this paper, $f\circ_0 g$ defines a graded commutative
  associative product and the Poisson bracket $\{f,g\}$ has the
  expression $\{f,g\}=f\circ_1 g-(-1)^{|f||g|} g\circ_1 f$.  The star
  product corresponds to a deformation quantization of this Poisson
  algebra.  We have the expected relations; e.g.,
  $\{f,g\}=\lim_{\hbar\to 0}\frac{f\star g-(-1)^{|f||g|}g\star
    f}{\hbar}.$
\end{remark}
In the rest of this section concerns the proof of Lemma
\ref{L:Weyl-product}, which establishes the coincidence of our
coordinate free definition of $W(V)$ with the coordinate dependent
description in terms of generators and relations, which may be
familiar to the reader (who is invited to skip the proof and move on
to Section \ref{S:properads}).

\begin{proof}[Proof of lemma \ref{L:Weyl-product}]
  We want to show that the $\hbar^k$ term of $gf$ using the above
  commutation relations is $g\circ_k f$. Using the relations to put
  the result of $gf$ back in standard form with $p$'s on the right is
  a process that involves commuting all the $p^m$'s in $g$ with the
  $q_j$'s in $f$.  As $p^m$ is moved to the right, each occurence of
  $p^mq_j$ is replaced by the two terms $q_jp^m$ and
  $\hbar\delta_j^m$.  We need to show that the signs and combinatorial
  factors are correct.

  Moving a variable $p^m$ to the right as far as possible involves the
  sum of moving it past all the $q_j$ with replacing it with
  $\hbar\delta^m_j$ as it passes each $q_j$.  This process induces a
  recursive sequence of choices, for each $p^m$, of moving all the way
  to the end or replacing with an $\hbar\delta^m_j$ on one of the
  remaining $q_j$.  A term with an $\hbar^k$ coefficient will come
  from the choice of $|\bar{m}|-k$ of the $p^m$ to move all the way to
  the end, with the remaining $k$ of the $p^m$ interacting with some
  $q_j$.  This further involves the choice of $k$ of the $q_j$ and a
  permutation of $\sS_k$ to govern which of the $k$ $p^m$ interacts
  with which of the $k$ chosen $q_j$.  Then the $\hbar^k$ term of the
  product $gf$, put in standard form, is the following sum over
  $\sigma\in \sS^{-1}_{m,m-k}$, $\tau\in \sS^{-1}_{j,k}$, and $\rho\in
  \sS_k$
$$
\sum\epsilon \delta^{\bar{m}_k}_{\rho\bar{j}_k} \
q_{\bar{n}}q_{\bar{j}_{j-k}} p^{\bar{m}_{m-k}}p^{\bar{i}}
$$
The signs $\epsilon$ will be reconciled at the end of the argument.

Let us evaluate the above expression on $[q_{\otimes \bar{v}}]\in
S^{m-k+i}$.  First, $q_{\otimes \bar{v}}$ is symmetrized, and then
$p^{\otimes \bar{m}_{m-k}}$ is evaluated on the first $m-k$ factors of
each summand of the symmetrization while $p^{\otimes \bar{i}}$ is
evaluated on the following $i$ factors of each summand.  Using the
unique representation of a permutation in $\sS_{m-k+i}$ as the
composition of an element of $\sS_{m-k}\times \sS_i$ with an
$(m-k+i,m-k)$-unshuffle, $gf([q_{\otimes\bar{v}}])$ is the following
sum over $ \pi\in \sS^{-1}_{m-k+i,m-k}$, $\eta\in \sS_{m-k}$,
$\theta\in \sS_i$, and $\sigma,\tau,\rho$ as before:
$$
\sum\epsilon \delta^{\bar{m}_k}_{\rho\bar{j}_k}
\delta^{\bar{m}_{m-k}}_{\eta \bar{v}_{m-k}} \delta^{\bar{i}}_{\theta
  \bar{v}_i} [q_{\otimes \bar{n}}\otimes q_{\otimes\bar{j}_{j-k}}]
$$

Now, let us evaluate $g\circ_k f$ applied to the same element
$[q_{\otimes \bar{v}}]$.  By definition,
$$
g\circ_k f=\binom{m+i-k}{i}\binom{j}{k}m!i!s (\iota s q_{\otimes
  \bar{n}} p^{\otimes\bar{m}}\iota s)\circ_k (\iota s q_{\otimes
  \bar{j}}p^{\otimes \bar{i}}\iota s)\iota
$$
To apply this to $[q_{\otimes\bar{v}}]$, we begin by symmetrizing
$q_{\otimes\bar{v}}$.  Again, it is more convenient to view this as an
unshuffle followed by a product of permutations from $\sS_{m-k}$ and
$\sS_i$.  So we write
$$\iota [q_{\otimes\bar{v}}]=\frac{1}{(m-k+i)!}\sum q_{\otimes \eta \bar{v}_{m-k}}\otimes q_{\otimes \theta\bar{v}_i}$$
Next the second factor is resymmetrized, which has no effect since it
is already symmetric, and then $\iota s q_{\otimes \bar{j}}p^{\otimes
  \bar{i}}$ is applied to it, yielding
$$\frac{m!i!}{(m-k+i)!}\sum \epsilon \delta^{\bar{i}}_{\theta\bar{v}_i}q_{\otimes \eta\bar{v}_{m-k}}\otimes \iota s q_{\otimes\bar{j}} 
$$
We view each summand permutation in the symmetrization of $\bar{j}$,
again, as the composition of product permutations with a
$(j,k)$-unshuffle.  We will sum over $\rho\in \sS_k$ and
$(j,k)$-unshuffles $\tau$, but incorporate the $\sS_{j-k}$
permutations with $\iota$ and $s$.  So this is
$$
\frac{(j-k)!}{j!(m-k+i)!}\sum\epsilon\delta^{\bar{i}}_{\theta\bar{v}_i}
q_{\otimes \eta\bar{v}_{m-k}}\otimes q_{\otimes\rho\bar{j}_k}\otimes
\iota s q_{\otimes \bar{j}_{j-k}}
$$
Applying $\iota s$ to symmetrize the first two factors corresponds to
first symmetrizing each one individually and then shuffling them with
an $(m,m-k)$-shuffle $\sigma^{-1}$.  Since they are both already
symmetric, this gives
$$
\frac{(j-k)!(m-k)!k!}{m!j!(m-k+i)!}\sum\epsilon\delta^{\bar{i}}_{\theta\bar{v}_i}\sigma^{-1}(q_{\otimes
  \eta\bar{v}_{m-k}}\otimes q_{\otimes\rho\bar{j}_k})\otimes \iota s
q_{\otimes \bar{j}_{j-k}}
$$
Applying $\iota s q_{\otimes \bar{n}}p^{\otimes \bar{m}}$ to the first
factor gives
\begin{multline*}
  \frac{(j-k)!(m-k)!k!}{m!j!(m-k+i)!}\sum\epsilon\delta^{\bar{i}}_{\theta\bar{v}_i}
  \delta^{\bar{m}}_{\sigma^{-1} \eta\bar{v}_{m-k}\otimes
    \rho\bar{j}_k}\iota s q_{\otimes \bar{n}}\otimes \iota s
  q_{\otimes
    \bar{j}_{j-k}}\\
  =
  \frac{(j-k)!(m-k)!k!}{m!j!(m-k+i)!}\sum\epsilon\delta^{\bar{i}}_{\theta\bar{v}_i}
  \delta^{\bar{m}_{m-k}}_{\eta\bar{v}_{m-k}}
  \delta^{\bar{m}_k}_{\rho\bar{j}_k} \iota s q_{\otimes
    \bar{n}}\otimes \iota s q_{\otimes \bar{j}_{j-k}}
\end{multline*}
By Lemma \ref{lemma:symtens1}, symmetrizing this whole expression
means that we can ignore the symmetrizations on $\bar{n}$ and
$\bar{j}_{j-k}$.  Including the combinatorial factor
$\binom{m+i-k}{i}\binom{j}{k}m!i!$, we obtain
$$
\sum\epsilon\delta^{\bar{i}}_{\theta\bar{v}_i}
\delta^{\bar{m}_{m-k}}_{\eta\bar{v}_{m-k}}
\delta^{\bar{m}_k}_{\rho\bar{j}_k} [q_{\otimes \bar{n}\otimes
  \bar{j}_{j-k}}]
$$
just as before.

Finally, we check equality of the signs.  $\epsilon(\pi,\bar{v})$,
$\epsilon(\eta,\bar{v}_{m-k})$, $\epsilon(\theta,\bar{v}_i)$,
$\epsilon(\tau,\bar{j})$, $\epsilon(\rho,\bar{j}_k)$, and the sign
$\epsilon(\sigma,\bar{m})$ are all on both the right and left-hand
side.  On the left side, there are also the signs
$(-1)^{|\bar{j}_{j-k}||\bar{m}_{m-k}|}$, $||\bar{m}_k||$, and
$||\bar{v}||$, the first from commuting the noninteracting $q_j$ and
$p^m$ past one another and the second and third the induced sign of
applying a tensor product of $p$'s to a tensor product of $q$'s. These
are not literally the correct signs, but they coincide whenever the
corresponding $\delta$ functions are nonzero.  On the right, there are
the signs $||\bar{i}||$ and $||\bar{m}||$ for the same reason, along
with the sign $(-1)^{|f||\bar{v}_{m-k}|}$ from applying $f$ to the
tensor factors on the right.  Expanding either side with the relations
$$||\bar{m}_{m-k}||\,||\bar{m}||\,||\bar{m}_k||=(-1)^{|\bar{m}_k||\bar{m}_{m-k}|}$$
$$||\bar{v}_{i}||\,||\bar{v}||\,||\bar{v}_{m-k}||=(-1)^{|\bar{v}_i||\bar{v}_{m-k}|}$$
along with noting that $|\bar{m}_k|=|\bar{j}_k|$,
$|\bar{m}_{m-k}|=|\bar{v}_{m-k}|$, and $|\bar{v}_i|=|\bar{i}|$ when
the corresponding $\delta$ functions are nonzero yields the equality
of the two signs.
\end{proof}

\section{Properads and coproperads}\label{S:properads}
The main reference for this section is \cite{V}.  A properad is,
roughly, an algebraic structure that models composable operations to
and from the tensor powers of a vector space.  In the same way that
operads govern gebras with only many to one operations, properads
govern gebras with many to many operations, such as Frobenius and
biLie algebras.  The dual notion of a properad is a coproperad, and
there are a number of ways to obtain a coproperad from a given
properad, and vice versa.  The most naive uses finite dimensional
pieces of a properad and dualizes each piece individually.  Starting
from the properad $\mathbf{P}$, this yields the coproperad
$\mathrm{co}\mathbf{P}$.  A more conceptually elegant method of
dualization is the bar or cobar construction.  The main result of this
paper, again, is that degree one elements of square zero in the Weyl
algebra $W(V)$, as discussed in the previous section, are in one to
one correspondence with $\cobar(\cofrob)$-algebra structures on $V$.

To describe algebras over $\cobar(\cofrob)$, we first define the
Frobenius coproperad $\cofrob$, then the cobar construction, and give
a presentation of the properad $\cobar(\cofrob)$.  Finally, we will
define algebras over a properad and obtain the relations on an algebra
over the particular properad in question.

\subsection{Preliminaries;  notation}
We now recall the notions of properad and coproperad, and algebras
over properads, cf. \cite{V}.
\begin{definition}
  A \emph{finite $n$-level directed graph $G$} consists of a triple
  $(\{V_i\},\{F_v\},\{\varphi_i\})$, given by the following data:
  \begin{enumerate}
  \item A finite ordered set $V_i$ of vertices on level $i$, for $i\in
    \{0,\ldots, n+1\}$.  $V_0$ and $V_{n+1}$ are called the incoming
    and outgoing vertices of the graph $G$, respectively.
  \item For each vertex $v\in \bigcup V_n$, two finite sets
    $F_v^{\text{in}}$ and $F_v^{\text{out}}$ of directed incoming and
    outgoing half-edges incident at $v$, with $|F^{\text{in}}_v|=0$
    and $|F^{\text{out}}_v|=1$ for $v\in V_0$, and
    $|F^{\text{in}}_v|=1$ and $|F^{\text{out}}_v|=0$ for $v\in
    V_{n+1}$. We denote by $F_v=F^{\text{in}}_v \sqcup
    F^{\text{out}}_v$ the disjoint unit of the incoming and outgoing
    half-edges.
  \item For $i\in \{0,\ldots, n\}$, a bijection
$$
\varphi_i:\bigcup_{v\in V_i} F_v^{\text{out}}\to \bigcup_{v\in
  V_{i+1}} F_v^{\text{in}}
$$ that joins outgoing half-edges of one level and incoming half-edges of the next.  $\varphi_0$ and $\varphi_n$ reorder the overall incoming and outgoing edges of the graph.
\end{enumerate}
\begin{figure}[h]
  \resizebox{!}{5cm}{
    \begin{pspicture}(-2,.7)(8.5,7.7)
      \psline[linestyle=dashed](0,7)(7,7)
      \psline[linestyle=dashed](0,5.5)(7,5.5)
      \psline[linestyle=dashed](0,4)(7,4)
      \psline[linestyle=dashed](0,2.5)(7,2.5)
      \psline[linestyle=dashed](0,1)(7,1) \rput(-.5, 7.5){level}
      \rput(-.5,7){$0$} \rput(-.5,5.5){$1$} \rput(-.5,4){$2$}
      \rput(-.5,2.5){$3$} \rput(-.5,1){$4$}
      \cnodeput[fillstyle=solid](1.5,7){v01}{}
      \cnodeput[fillstyle=solid](3.5,7){v02}{}
      \cnodeput[fillstyle=solid](5.1,7){v03}{}
      \cnodeput[fillstyle=solid](1.5,5.5){v11}{}
      \cnodeput[fillstyle=solid](4.3,5.5){v12}{}
      \cnodeput[fillstyle=solid](5.9,5.5){v13}{}
      \cnodeput[fillstyle=solid](1.1,4){v21}{}
      \cnodeput[fillstyle=solid](2.7,4){v22}{}
      \cnodeput[fillstyle=solid](4.3,4){v23}{}
      \cnodeput[fillstyle=solid](5.9,4){v24}{}
      \cnodeput[fillstyle=solid](1.5,2.5){v31}{}
      \cnodeput[fillstyle=solid](3.5,2.5){v32}{}
      \cnodeput[fillstyle=solid](5.1,2.5){v33}{}
      \cnodeput[fillstyle=solid](3.5,1){v41}{}
      \cnodeput[fillstyle=solid](5.1,1){v42}{}
      \nccurve[angleA=270, angleB=90]{v01}{v11} \nccurve[angleA=300,
      angleB=60]{v13}{v24} \nccurve[angleA=290, angleB=110]{v02}{v12}
      \nccurve[angleA=250, angleB=70]{v03}{v12} \nccurve[angleA=270,
      angleB=90]{v13}{v24} \nccurve[angleA=240, angleB=120]{v13}{v24}
      \nccurve[angleA=270, angleB=90]{v32}{v41} \nccurve[angleA=270,
      angleB=90]{v33}{v42} \nccurve[angleA=240, angleB=90]{v11}{v21}
      \nccurve[angleA=300, angleB=140]{v11}{v23} \nccurve[angleA=290,
      angleB=70]{v12}{v23} \nccurve[angleA=210, angleB=30]{v23}{v31}
      \nccurve[angleA=210, angleB=70]{v12}{v22} \nccurve[angleA=250,
      angleB=30]{v12}{v22} \nccurve[angleA=280, angleB=110]{v21}{v31}
      \nccurve[angleA=240, angleB=70]{v22}{v31} \nccurve[angleA=260,
      angleB=70]{v23}{v32} \nccurve[angleA=310, angleB=120]{v22}{v33}
    \end{pspicture} }
  \caption{A finite $3$-level directed graph \label{fig2} }
\end{figure}
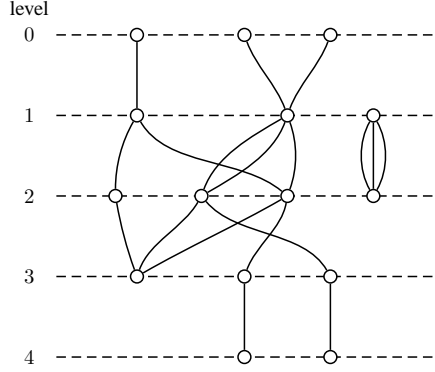
Two graphs $(\{V_i\},\{F_v\},\{\varphi_i\})$ and
$(\{U_j\},\{G_u\},\{\psi_j\})$ are equivalent if there are
order-preserving bijections on the vertices on each level and
bijections of the incoming and outgoing-half-edges which respect the
joining bijections $\varphi$ and $\psi$.  A $L,R$ labelling of a graph
is a pair of bijections from the set $L$ to the incoming level one
half-edges, and a bijection from the outgoing level $n$ half-edges to
the $R$.

The set of finite $n$-level directed graphs up to equivalence is
denoted by $\mathscr{G}^{(n)}$.
\end{definition}
\begin{definition}
  The \emph{geometric realization} of a graph
  $(\{V_i\},\{F_v\},\{\varphi_i\})$ is the topological space, defined
  as the quotient of the disjoint union
$$\Big(\amalg_{v\in \bigcup V_i} *_v \Big)\sqcup \Big(\amalg_{f\in \bigcup_v F_v} I_f\Big), $$ where $*_v$ denotes a one point space and $I_f$ denotes a copy of the unit interval $[0,1]$, divided by the equivalence relation generated by 
\begin{enumerate}\item $0_f\sim *_v$ if $f\in F_v$.
\item $1_{f_1}\sim 1_{f_2}$ if $\varphi_i(f_1)=f_2$ for some $i$.
\end{enumerate}
$G$ is called \emph{connected} if its geometric realization is
connected.  The set of finite connected $n$-level directed graphs with
$k$ incoming and $\ell$ outgoing edges is denoted
$\mathscr{G}^{(n)}_c(k,\ell)$, and let
$\mathscr{G}^{(n)}_c=\sqcup_{k,\ell} \:\mathscr{G}^{(n)}_c(k,\ell)$.
\end{definition}

An \emph{$\sS$-bimodule} in the category of graded vector spaces
(chain complexes) consists of a set of graded vector spaces (chain
complexes) $\{P(m,n)\}$ for $m,n\ge 0$ with commuting left $\sS_m$ and
right $\sS_n$ actions.  The category of $\sS$-bimodules is denoted by
$\mathscr{C}$. There is a functor
$$\boxtimes_c:\mathscr{C}\times \mathscr{C}\to \mathscr{C}$$ which acts on two $\sS$-bimodules $P$ and $Q$ by taking
$$
P\boxtimes_c
Q(k,\ell)=\bigoplus_{\mathscr{G}^{(2)}_c(k,\ell)}\bigotimes_{v\in V_2}
P(|F_v^{\text{out}}|,|F_v^{\text{in}}|)\otimes \bigotimes_{v\in V_1}
Q(|F_v^{\text{out}}|,|F_v^{\text{in}}|)\Big/ \sim,
$$
where $|X|$ denotes the number of elements of a finite set $X$, and
the equivalence relation consists of the following two parts. For one,
we divide out by
\begin{multline*}
  (\bigotimes p_i\otimes \bigotimes
  q_j)_{\{\{V_i\},\{F_v\},\{\varphi_0,\varphi_1,\varphi_2\}\}} \\
  \sim (\bigotimes \sigma_i p_i \rho_i \otimes \bigotimes \tau_j q_j
  \eta_j)_{\{\{V_i\},\{F_v\},\{\varphi_0(\prod
    \eta^{-1}),(\prod\tau^{-1}) \varphi_1(\prod \rho_i),(\prod
    \sigma_i) \varphi_2\}\}}
\end{multline*}
This construction does not have the appropriate $\sS$-bimodule
structure, so we must tensor over $\prod_{v\in V_2}
\sS_{v^{\text{out}}}$ with $\prod_{v\in V_3} \sS_{v^{\text{in}}}$ and
similarly with the incoming. The other equivariance relation is
\begin{multline*}
  (\bigotimes p_i\otimes \bigotimes
  q_j)_{\{\{V_i\},\{F_v\},\{\varphi_0,\varphi_1,\varphi_2\}\}} \\
  \sim (\sigma^{-1} \left(\bigotimes p_i\right) \otimes
  \tau\left(\bigotimes q_j\right) )_{\{\{V_i\},\{F_v\},\{\tau^{-1}
    \varphi_0,\sigma \varphi_1\tau, \varphi_2\sigma^{-1}\}\}}
\end{multline*}
where the action of $\tau$ and $\sigma$ in the compositions with the
$\varphi$ should be taken as acting on blocks of size equal to the
number of outputs or inputs of $p_i$ or $q_j$ as appropriate.  The
actions on the $p_i$ and $q_j$ themselves have signs as usual.

In words, $P\boxtimes_c Q(k,\ell)$ consists of connected two-level
graphs with elements of $Q$ labelling the vertices on the first level
and elements of $P$ labelling the vertices on the second level.  The
labelling elements should be chosen from the pieces $P(k',\ell')$ so
that $k'$ is the number of incoming flags at the vertex and $\ell'$
the number of outgoing flags.

\begin{definition}Let $I$ be the $\sS$-bimodule which has $I(1,1,0)=k$
  and $I(n,m,\chi)=0$ otherwise.
\end{definition}

The functor $\boxtimes_c$, along with the identity object $I$, makes
$\mathscr{C}$ a monoidal category.  This means that there is a natural
transformation expressing the associativity of $\boxtimes_c$ and two
more expressing that $I$ is a left and right identity for
$\boxtimes_c$.

\begin{definition}
  A \emph{properad} $\mathbf{P}$ is a monoid in the category
  $\mathscr{C}$.  This data comprises two morphisms:
  \begin{enumerate}
  \item A composition morphism $\mu:\mathbf{P}\boxtimes_c
    \mathbf{P}\to\mathbf{P}$, and
  \item A unit morphism $\iota:I\to \mathbf{P}$.
  \end{enumerate}
  Composition must satisfy associativity up to the natural
  transformation for associativity of $\boxtimes_c$ as well as left
  and right unit properties (e.g., $\mu\circ (\iota\boxtimes \id)\sim
  \id$ via the natural transformation between $I\boxtimes \mathbf{P}$
  and $\mathbf{P}$).
\end{definition}
\begin{definition}
  A \emph{coproperad} $\mathbf{C}$ is a comonoid in the category
  $\mathscr{C}$. This data again comprises two morphisms:
  \begin{enumerate}
  \item A decomposition morphism $\Delta:\mathbf{C}\to
    \mathbf{C}\boxtimes_c\mathbf{C}$, and
  \item A counit morphism $\eta:\mathbf{C}\to I$.
  \end{enumerate}
  Decomposition must satisfy coassociativity (up to the natural
  transformation for associativity of $\boxtimes_c$) as well as left
  and right counit properties dual to the unit properties.
\end{definition}
\begin{example}
  If $(V,d)$ is a chain complex (whit differential of degree $|d|=1$),
  then $T^kV$ has the induced structure of a chain complex where
  $d(\bar{v})=dv_1\otimes v_2\otimes\cdots\otimes
  v_k+\cdots+(-1)^{|v_1|+\cdots+|v_{i-1}|}v_1\otimes\cdots\otimes
  v_{i-1}\otimes dv_i\otimes\cdots\otimes v_k+\cdots$.  If $(V,d)$ and
  $(V',d')$ are chain complexes, then $\hom(V,V')$ has the induced
  structure of a chain complex with differential $f\mapsto d' f -
  (-1)^{|f|}fd$.  Thus, if $(V,d)$ is a chain complex, then
  $\End(V)(m,n):=(\hom(T^mV,T^nV),D)$ is a chain complex, where $D$ is
  the induced differential.  There are commuting left $\sS_m$ and
  right $\sS_n$ actions and the obvious composition maps, so $\End(V)$
  is a properad.  Note that in the graded context, the symmetric
  actions respect the grading, so that, for example,
  $\psi\sigma(\bar{v})=\psi(\sigma\bar{v})$.
\end{example}
\begin{definition}
  By definition, $(V,d)$ has the structure of an \emph{algebra} over
  the properad of chain complexes $\mathbf{P}$, if there is a properad
  morphism $\mathbf{P}\to \End(V)$.

  Explicitly, this means that there are degree zero maps
  $\mathbf{P}(m,n)\to \hom(T^mV,T^nV)$ which are equivariant with
  respect to both the $\sS_m$ and $\sS_n$-actions, and such that
  composition in $\mathbf{P}$ corresponds to actual composition of
  maps between tensor powers of $V$.  Furthermore, the differential
  $d$ in $\mathbf{P}(m,n)$ corresponds to the differential $D$ in the
  Hom complex.
\end{definition}

\subsection{The $\cofrob$ coproperad}\label{SS:cofrob}
We define an object $\{\cofrob(m,n,\chi)\}$ in the category
$\mathscr{C}$ of $\sS$ bimodules and morphisms $\eta:\cofrob\to I$ and
$\Delta:\cofrob \to \cofrob \boxtimes_c \cofrob$ as follows:
\begin{enumerate}
\item For $m$, $n\ge 1$ and $\chi\ge m+n-2$ and of the same parity as
  $m+n$, we set $\cofrob(m,n,\chi)=k$.  This corresponds to the unique
  $m$ to $n$ Frobenius algebra operation of ``genus''
  $\frac{\chi-m-n}{2}$.  For all other choices of $m$, $n$, and
  $\chi$, $\cofrob(m,n,\chi)=0$.
\item All the $\sS_m$ and $\sS_n$ actions are trivial.
\item The map $\eta$ is projection onto the factor $\cofrob(1,1,0)$.
\item The map $\Delta$ is more involved to describe, and will be done
  below.
\end{enumerate}
We first examine $\cofrob\boxtimes_c\cofrob$.  This consists of all
connected two-level trees labeled by elements of $\cofrob$ of the
appropriate grading, up to equivalence.  Since the symmetric group
actions are trivial, only the information of the number of edges
between two vertices is important in a two-level graph, but not the
actual combinatorics of how the flags are connected.  Therefore, a
two-level tree with $m$ inputs and $n$ outputs marked with elements of
$\cofrob$ up to equivalence consists of:
\begin{enumerate}
\item Partitions of $\{1,\ldots, m\}$ and $\{1,\ldots, n\}$ into
  nonempty sets $u_i$ and $v_j$, where $u_i$ denotes the vertices on
  the first level and $v_j$ the vertices on the second level.  This is
  taken up to reordering of the vertices, with the induced sign.
\item For each pair $(u,v)$ from $V_1\times V_2$, a nonnegative number
  $e(u,v)$ of edges from $u$ to $v$ so that the total number of edges
  $e(u)=\sum_v e(u,v)$ and $e(v)=\sum_u e(u,v)$ are positive.
\item A weight $\chi$ for each $u$ which is of the same parity and at
  least $|u|+e(u)-2$, and likewise for $v$.
\end{enumerate}
Furthermore, the geometric realization of the graph must be connected.
Then the decomposition map $\Delta$ takes $\cofrob(m,n,\chi)\cong k$
into the direct sum over such two level graphs of a tensor product of
$\cofrob(m',n',\chi')$.  It is just the zero map on any zero summand
and a combinatorial factor $\eta_G$ times the canonical isomorphism of
$k$ with $k^{\otimes i}$ on the summand spanned by a graph $G$ where
each factor of the tensor product is $k$.  We define the combinatorial
factor $\eta_G$ as the product over pairs $(u,v)$ of vertices from
$V_1\times V_2$ of $\frac{1}{e(u,v)!}$.

\begin{remark}$\cofrob$ can be interpreted in some sense as the naive
  dual of the Frobenius properad or as the Koszul dual of a
  commutative, rather than skew, version of the involutive biLie
  properad.  We thought it more expedient to define it directly,
  rather than introduce an additional level of duality.
\end{remark}

\begin{proposition}The data $(\cofrob, \Delta, \eta)$ defines a
  coproperad.
\end{proposition}
\begin{proof}
  We have to check coassociativity for $\Delta$, and the left and
  right counit properties for $\eta$.  To see that $\Delta$ is
  coassociative, consider $\cofrob^{\boxtimes_c 3}$.  This is the
  vector space spanned by three-level graphs marked by $\cofrob$. Let
  edges between the first and second level of vertices generate an
  equivalence relation on vertices; then let the equivalence classes
  be the top level of vertices of a new graph, with incoming flags the
  disjoint union of the incoming flags of the constituent vertices in
  the upper level of the equivalence class and outgoing flags the
  disjoint union of the outgoing flags of the constituent vertices in
  the lower level of the equivalence class.  Let the grading of an
  equivalence class be the sum of the gradings of its member vertices.
  Let the third level of vertices of the original graph be the bottom
  level of vertices of this new graph; then the old (three level)
  graph is part of the image of the new (two level) one under
  $\Delta\boxtimes_c Id$.  If the original graph is $G$, call this
  graph $G_{12}$.

  Given a vertex $[v]$ in the first level of $G_{12}$, that is, an
  equivalence class of vertices of $G$, we construct a two level graph
  marked by $\cofrob$ denoted $H_v$.  Let the vertices on the first
  and second levels of $H_v$ be the vertices of $G$ in $[v]$; let the
  incoming flags, the vertex weights, and the edges between the first
  and second levels be induced by the corresponding data in $[v]$.
  Let the number of outgoing flags be determined by $[v]$; however
  $[v]$ does not induce a labelling, so choose an arbitrary labelling
  for the outgoing vertices.  Intuitively, $H_v$ represents $[v]$ as
  an independent graph.

  A similar construction can be performed for the second and third
  level of the graph $G$ and will yield a two-level graph $G_{23}$
  which has the old graph as part of its image under $Id\boxtimes_c
  \Delta$.  We similarly get $H_v$ for $[v]$ in the second level
  vertex set of $G_{23}$.

  Both $G_{12}$ and $G_{23}$ are part of the image under $\Delta$ of
  the graph $G_{123}$ obtained from the original by collapsing all of
  the vertices and internal edges to a single vertex.  Let $\pi_G$
  denote the linear projection onto the one dimensional subspace
  spanned by $G$.  Then $\pi_{G}(\Delta \boxtimes_c Id)\circ \Delta
  [G_{123}]$ is equal to $\pi_G (\Delta\boxtimes c Id) \circ
  \pi_{G_{12}}\circ \Delta [G_{123}]$ because no other two level
  graphs can yield $G$ under expansion of the vertices on the first
  level.  The cognate statment is true for $G_{23}$.

  So to show coassociativity, it is enough to show that for a marked
  three-level graph $G$,
$$
\pi_G \circ (\Delta \boxtimes_c Id) \circ \pi_{G_{12}}\circ \Delta
[G_{123}]= \pi_G \circ (Id\boxtimes_c \Delta)\circ \pi_{G_{23}}\circ
\Delta [G_{123}].
$$
No signs are introduced in either of the applications of $\Delta$, so
in order for this equality to be true, it is only necessary that the
combinatorial factors agree.  If $V_{12}$ is the vertex set of the
first level of $G_{12}$, the level consisting of equivalence classes,
and likewise $V_{23}$, then the above equality is, at the level of
combinatorial factors,
$$
\eta_{G_{12}}\prod_{[v]\in V_{12}}\rho_v\eta_{H_v}=
\eta_{G_{23}}\prod_{[v]\in V_{23}}\rho_v\eta_{H_v}
$$
where for $[v]$ in $V_{12}$, $\eta_{H_v}$ is the product
$\frac{1}{e(u,w)!}$ for $u,w$ in $[v]$, and $\rho_v$ counts the number
of two-level graphs which are similar enough to the graph $H_v$ that
the projection of $\Delta [v]$ on the summand spanned by them
contributes to this projection on the $G$-summand.

The product of the $\eta_{H_v}$ over $[v]$ in $V_{12}$ is the product
of $\frac{1}{e(u,w)!}$ over all pairs of vertices from the first and
second levels of $G$; for pairs where the two vertices come from
different equivalence classes, $e(u,w)$ must be zero, so the
contribution from such pairs is $1$.  $\eta_{G_{12}}$ is the product
of $\frac{1}{e([v],z)!}$ for $[v]$ in $V_{12}$ and $z$ in the third
level of $G$, where $e([v],z)=\sum_{w\in [v]}e(w,z)$.

To see this equality, consider $G_{12}$.  Fix a labelling on the
incoming flags of the second level vertices.  Then there is some
finite number $\rho_v$ of relabellings of the outgoing flags of $H_v$
which are compatible with the given labelling, in the sense that if
such relabellings are chosen for each $[v]$, then connecting the
relabelled $H_v$s along the identity permutation to the labelled
incoming flags of the second level vertices of $G_{12}$ yields a graph
isomorphic to $G$ as a three-level graph with vertices marked by
$\cofrob$.

To justify the notation, the $\rho_v$ must be independent of one
another; this occurs because distinct $[v]$ correspond to distinct
subsets of the incoming flags of $G$ so that each incoming flag of the
second level of $G_{12}$ must be connected to a unique $[v]$.  So the
outgoing flags from each relabelled $H_v$ can be considered
seperately, meaning the equality is well-defined.

It remains to calculate $\rho_v$.  This counts the number of ways of
relabelling the outgoing flags of $H_v$ to be consistent with the
incoming flags of the third level vertices of $G$.  By equivalence and
by the trivial symmetric action on a vertex $w$ in the second level of
$G$ in $[v]\in V_{12}$, any relabelling is equivalent to one where the
order of the outgoing flags at $w$ respects a fixed order of the third
level vertices of $G$.

Now consider a vertex $z$ on the third level of $G$ and a vertex
$[v]\in V_{12}$.  To be consistent, a relabelling must associate the
incoming flags of $z$ associated to $[v]$ to the specific outgoing
flags of the constituent $w$ determined by the order in the previous
paragraph.  Two relabellings from $[v]$ to $z$ are equivalent if they
differ only by a permutation of the outgoing flags of $w$.  Also, if
there is an isomorphism of $G$ that interchanges $w$ and $w'$, then
two relabellings interchanging the labels of their outgoing flags are
equivalent.

Then we are counting partitions of $e([v],z)$ into pieces of size
$e(w,z)$, up to simultaneous relabelling of the partitions
corresponding to $w$ and $w'$ for all $z$ if there is an isomorphism
of $G$ interchanging them.  The number of ordered partitions is
determined by a familiar combinatorial formula:
$$
\frac{e([v],z)!}{\prod_{w\in [v]} e(w,z)!}.
$$
So the number of relabellings $\rho_v$ is the product of these factors
for all $z$ divided by permutations of second level vertices along
isomorphisms of $G$.  Suppose the vertices on the second level of
$H_v$ are divided into equivalence classes $W_1,\ldots, W_r$, where
$w$ and $w'$ are in the same equivalence class if there is an
isomorphism of $G$ interchanging them.  Note that if there is an
isomorphism interchanging any two vertices on the second level of $G$,
then they must be in the same equivalence class in $V_{12}$ and in
$V_{23}$.  Then we obtain
$$
\rho_v=\frac{\prod_z e([v],z)!}{\prod_{w\in [v],z} e(w,z)! \prod_1^r
  |W_i|!}.
$$

Now the left hand side of the equality that will prove coassociativity
is
\begin{multline*}
  \prod_{([v],z)}\frac{1}{(e([v],z)!}\prod_{[v]}\frac{\prod_ze([v],z)!}{\prod_{w\in
      [v],z} e(w,z)!\prod_1^r |W_i|!}\prod_{u,w}\frac{1}{e(u,w)!}  \\=
  \prod_{w,z}\frac{1}{e(w,z)!}\prod_{u,w}\frac{1}{e(u,w)!}\prod_{W_i}\frac{1}{|W_i|!}
\end{multline*}
where the products are taken over pairs $w,z$ from the second and
third levels of $G$, pairs $u,w$ from the first and second levels of
$G$, and all equivalence classes of second level vertices of $G$.

A similar calculation shows that the right hand side is the same,
showing coassociativity.

To see that $\cofrob$ is counital, note that one factor of the
decomposition of any element $x$ of $\cofrob$ is the two-level graph
with $x$ on top and only copies of $\cofrob(1,1,0)$ on the bottom.
Applying $\id\boxtimes_c \eta$ to this yields $x$.  On the other hand,
any other factor of the decomposition will have something other than
$\cofrob(1,1,0)$ on the bottom, and $\id\boxtimes_c \eta$ will yield
$0$.  A similar argument applies for the left counit property.
\end{proof}

\subsection{The cobar construction}\label{SS:cobar}
Next it is necessary to discuss the cobar construction, which begins
with a coproperad $\mathbf{C}$ and generates a properad
$\cobar(\mathbf{C})$; cf. \cite[section 4]{V}.  This properad is
freely generated on the constituent spaces of the associated
$\sS$-module $\bar{\mathbf{C}}[-1]$, which in this context can be
interpreted as $\mathbf{C}_{m,n,g}/\mathbf{C}_{1,1,0}$ with a shift in
grading, putting all the generators in degree negative one.

This free generation is under properadic composition and the symmetric
group actions (subject to the associativity and equivariance
relations), as a properad of graded vector spaces.  The decomposition
maps $\Delta_{m',n'}^{k,g'}$ enter the picture in terms of a
differential $d$ on $\cobar(\mathbf{C})_{m,n,g}$ which makes this into
a properad of chain complexes.

A generic basis element of the free properad on an $\sS$-module $V$ is
a tree labelled by elements of $V$.  So fixing an order on the
vertices of the tree, and on the edges connecting two vertices, it is
a tensor product of elements from $V(m,n)$.  Specifying an element
with homogeneous grading, it is a tensor product of elements from
$V(m,n,\chi)$.  The differential acts on this space as a derivation,
meaning that up to sign, it is determined by its action on $V$ itself:
$$
d(v_1\otimes\cdots v_k)=dv_1\otimes\cdots
v_k+\cdots+(-1)^{|v_1|+\cdots+|v_{k-1}|}v_1\otimes \cdots dv_k
$$
The differential acts on $V$ as a restriction of the decomposition map
$\Delta$.  Call vertices in a graph labelled with the identity trivial
vertices (in this case this is any vertex with $m=n=1$ and $\chi=0$).
There is a quotient map on $V\boxtimes_c V$ which kills any graph with
more or fewer than two nontrivial vertices.  Note that because the
grading of the identity map is even, we can also forget the ordering
on the vertices on each level, as their permutation will not introduce
a sign.  The composition of this quotient with decomposition gives the
action of $d$ on $V$ in the cobar construction.  Coassociativity and
the shift in the grading guarantee that $d^2=0$.

\subsection{The properad $\cobar(\cofrob)$}
Now we describe the properad $\cobar(\cofrob)$.  First, without the
differential, it is just the free properad on the reduced version
$\overline{\cofrob}$, i.e. an element of the $(r,t,\chi)$ piece is a
connected properad composition of elements of $\overline{\cofrob}$ of
grading $(r_i,t_i,\chi_i)$ with total grading $(r,t,\chi)$ under the
rules for the composition.

The only relations, other than those of equivariance and
associativity, are those imposed by $d$.  Thus, we need to determine
how $d$ acts on $\cobar(r,t,\chi)$.  Its image is contained in
two-level graphs with appropriate total grading and only one
nontrivial vertex on each level.  The $r$ inputs and $t$ outputs need
to be divided between the two non-trivial vertices.  There needs to be
some positive number of output flags from the first vertex connected
to input flags from the second.  Finally, any remaining grading must
be shared between the two vertices. Therefore, we take a sum over
$1\le i\le r$, $1\le k\le \frac{1}{2}(\chi-m-n)+2$, $k\le j\le t+k-1$,
$i+j\le \chi_1\le \chi-2k$, $(r,r-i)$ shuffles $\tau$, $(t,t-j)$
unshuffles $\sigma$, along with $m$, $n$, and $\chi_2$ which are
induced as $i+m-k=r$, $j+n-k=t$, $\chi_1+\chi_2=\chi$.  Using this
sum, we have
$$
d(\mathbf{1}_{r,t,\chi})=\sum \frac{1}{k!}\tau (\mathbf{1}_{m,
  n,\chi_2}\otimes \mathbf{1}_{i,j,\chi_1})\sigma
$$
The bounds on $i,j,k,\chi_1$ ensure that all of the indices here have
the appropriate size.  If $\chi_1$ or $\chi_2$ has the wrong parity
then the term is zero.  Since the order of the vertices on each level
doesn't matter and the symmetric actions are trivial, we can fix a
convention without introducing signs, namely that on the first level,
all of the trivial vertices precede the nontrivial vertex; on the
second level the nontrivial vertex precedes the trivial ones.
 
At this point it is convenient to regrade by ``genus'' instead of by
``Euler characteristic.''  This means that we replace the grading
$\chi$, which is at least $m+n-2$ and of the same parity as $m+n$ with
$g=\frac{1}{2}(\chi+2-m-n)$, which is then just nonnegative.  With
this regrading, properadic composition of two vertices along $k$ flags
has degree $k-1$ instead of $0$.  Rewriting $d$ with this grading we
get
$$
d(\mathbf{1}_{r,t,g})=\sum \frac{1}{k!}\tau (\mathbf{1}_{m,n,
  g_2}\otimes \mathbf{1}_{i,j, g_1})\sigma
$$
where $1\le k\le g+1$ and $0\le g_1\le g+1-k$, while $i$, $j$,
$\sigma$, and $\tau$ are taken over the same range as before.  Now
$g_1+g_2+k-1=g$.

\subsection{Algebras over $\cobar(\cofrob)$}\label{algebras}
We now state and prove our main theorem.
\begin{theorem}\label{T:main}
  There is a one to one correspondence between algebra structures over
  $\cobar(\cofrob)$ on $V$ and elements $H$ of degree $-1$ in $W(V)$
  such that $H\star H=0$.
\end{theorem}

\begin{proof}
  The properad $\cobar(\cofrob)$ is quasifree, meaning that every
  relation among two or more generators involves $d$.  These relations
  were summarized above.  Therefore the structure of a
  $\cobar(\cofrob)$-algebra on $V$ is equivalent to a collection of
  graded symmetric maps $\varphi_{r,t,g}:T^mV\to T^nV$ (with no
  $\varphi_{1,1,0}$) which satisfy the relations above.  We can define
  $\tilde{\varphi}_{r,t,g}:S^rV\to S^tV$ as $\tilde{\varphi}_{r,t,g}:=
  s\varphi_{r,t,g}\iota$, where $s$ and $\iota$ are the maps from
  section \ref{S:SV-TV-review}.

  Because the $\varphi_{r,t,g}$ are symmetric, they can be recovered
  from $\tilde{\varphi}_{r,t,g}$ as
  $\varphi_{r,t,g}=\iota\tilde{\varphi}_{r,t,g} s$.  This can be seen
  as follows:
$$
\iota\tilde{\varphi}_{r,t,g} s (\bar{v})= \frac{1}{r!}(\iota s)
\varphi_{r,t,g}(\sum_{\sigma\in \sS_r} \sigma \bar{v})
=\frac{1}{r!}\sum_{\sigma\in \sS_r}(\iota s) \varphi_{r,t,g}(\sigma
\bar{v}) =(\iota s) \varphi_{r,t,g}(\bar{v})$$ Since $\sigma$ applied
to $\varphi_{r,t,g}(\bar{v})\in T^tV$ is $\sigma
\varphi_{r,t,g}=\varphi_{r,t,g}$, $(\iota s)$ is the identity on
$\varphi_{r,t,g}(\bar{v})$.

Now let us examine the relations involved in a
$\cobar(\cofrob)$-algebra.  This is a structure consisting of a degree
$-1$ differential $d$ and a collection of degree $-1$ maps
$\varphi_{r,t,g}:T^rV\to T^tV$ along with a symmetry condition, which
can be expressed by saying that they come from the symmetric maps
$\tilde{\varphi}_{r,t,g}$ instead.  These maps are subject to
coherence relations.  All these relations involve only
$D\varphi_{r,t,g}$ and compositions of two $\varphi_{r,t,g}$ indexed
by a two-vertex tree with $k$ edges between the two vertices.

$$  D(\varphi_{r,t,g})=\sum \frac{1}{k!}\tau (\varphi_{m,n, g_2}\circ_k
\varphi_{i,j,g_1})\sigma $$ But $ D\varphi(\bar{v})=
d(\varphi(\bar{v}))+\varphi(d\bar{v})$, where $d$ here is extended as
a derivation $(d\otimes \id\otimes\cdots)\pm (\id\otimes d\otimes
\cdots)\pm\cdots$.  This is $d\circ_1 \varphi+\varphi \circ_1 d$, so
defining $\varphi_{1,1,0}=-d$, the relations are precisely
$$
\sum \frac{1}{k!}\tau
(\varphi_{m,n,g_2}\circ_k\varphi_{i,j,g_1})\sigma=0.
$$

Now, let $(V,\{\varphi_{r,t,g}\})$ be an algebra over
$\cobar(\cofrob)$.  Define $H\in W(V)$ as $\bigoplus
\frac{1}{t!}\tilde{\varphi}_{r,t,g}\hbar^g$.  Then the $\hbar^g$ part
of $\hom(S^r V,S^t V)$ in $H\star H$ is
$$\sum
\frac{1}{n!j!}\tilde{\varphi}_{m,n,g_2}\circ_k\tilde{\varphi}_{i,j,g_1},
$$ where the sum ranges over $m+i-k=r$, $n+j-k=t$, and
$g_1+g_2+k-1=g$.  If this is applied to $[\bar{v}]$, then its
injective image under $\iota$ is equal to
$$\sum \frac{1}{k!t!}\tau(\varphi_{m,n,g_2}\circ_k\varphi_{i,j,g_1})\sigma=0.$$
This shows that a $\cobar(\cofrob)$-algebra defines an element of
square zero in the Weyl algebra.  On the other hand, suppose that we
have such an element $H$ of square zero in the Weyl algebra of a
graded vector space $V$.  Then $(H_1^{1\,(0)})^2=0$, so we can take it
to be a differential $d$ on $V$.  Then, by definining
$\varphi_{r,t,g}=n!H_r^{t\, (g)}$, the reverse equality holds, namely,
$$\sum \frac{1}{k!}\tau (\varphi_{m,n,g_2}\circ_k\varphi_{i,j,g_1})\sigma
=t!\iota (\sum H_{m}^{n\,(g_2)}\circ_k H_{i}^{j\,
  (g_1)})[\bar{v}]=t!\iota H\star H (\bar{v})=0.$$
\end{proof}

\section{The homology of algebras over
  $\cobar(\cofrob)$}\label{S:homology}
The homology of a properad is again a properad and if $V$ is an
algebra over a properad $\mathbf{P}$, then $HV$ is an algebra over the
properad $H\mathbf{P}$.  To see this, recall from Section
\ref{algebras} that an algebra $V$ over a properad $\mathbf{P}$ is a
collection of chain maps satisfying equivariance and compatibility
with composition from $\mathbf{P}(m,n)$ to $Hom(T^mV,T^nV)$.  The
induced maps on homology still satisfy equivariance and compatibility,
so that there is a properad morphism from $H\mathbf{P}(m,n)$ to
$HHom(T^mV,T^nV)$.  There is a natural isomorphism $HHom(T^mV,T^nV)\to
Hom(T^mHV,T^nHV)$, hence a properad morphism $H\mathbf{P}(m,n)\to
Hom(T^m HV, T^n HV)$ affording $HV$ with the structure of an algebra
over $H\mathbf{P}$.

For the properad $\cobar(\cofrob)$, grading by genus one identifies
symmetric generators $\mu\in \cobar(\cofrob)(2,1,0)$ and $\Delta\in
\cobar(\cofrob)(1,2,0)$ which are closed under the differential
because their decomposition is trivial in $\cofrob$.  By general
arguments on the cobar construction, $\mu$ and $\Delta$ can be seen
not to be boundaries, and therefore pass to nonzero classes $[\mu]$
and $[\Delta]$ in homology.  Considering the boundaries of the
generators in the $(3,1,0)$, $(1,3,0)$, $(2,2,0)$, and $(1,1,1)$
spaces of $\cobar(\cofrob)$ we see that in homology, $[\mu]$ satisfies
the Jacobi relation
$$[\mu]\circ_1 [\mu] (1+\sigma+\sigma^2)=0
$$
or, rewritten with $\mu$ as a bracket, more familiarly, this is
$$
[[a,b],c] + (-1)^{(|b|+|c|)|a|}[[b,c],a] +
(-1)^{(|a|+|b|)|c|}[[c,a],b]]
$$
$[\Delta]$ satisfies the coJacobi relation
$$(1+\sigma+\sigma^2)[\Delta]\circ_1 [\Delta]=0
$$
and $[\mu]$ and $[\Delta]$ together satisfy the five term
compatibility relation
$$[\Delta]\circ [\mu]+ (1+\tau)[\mu]\circ_1 [\Delta] (1+\tau)=0
$$
or, applied to $a\otimes b$,
\begin{multline*} [\Delta][a,b] + (-1)^{|a|} ([\mu]\otimes
  \id)a\otimes [\Delta]b +
  (-1)^{|a||b| + |b|} ([\mu]\otimes \id) b\otimes [\Delta]a \\
  +(-1)^{|a||b|} (\id\otimes[\mu])b\otimes [\Delta]a + (\id\otimes
  [\mu])[\Delta]a\otimes b
\end{multline*}
and the involutivity relation
$$[\mu]\circ [\Delta] = 0
$$

This shows that the homology $HV$ of a $\cobar(\cofrob)$-algebra $V$
is a (commutative as opposed to skew-commutative) involutive biLie
algebra, but we have not argued that our computation of the homology
is complete.  We conjecture that the homology of the properad
$\cobar(\cofrob)$ is the involutive bi Lie properad, but at present we
do not have a proof---there remains the possibility that there are
additional nonzero homology operations.

\appendix
\section{Combinatorial factors in detail}\label{A:combinatorics}
In this appendix, we collect some properties of symmetrization $\iota$
and projection $s$, and using this, we relate the partial composition
map for the tensor algebra with the one for the symmetric algebra.

The first two lemmas concern the effect of symmetrization part of a
vector in the tensor algebra.  The first asserts that the outcome of
symmetrizing part of a vector followed by symmetrizing the entire
vector is the same as simply symmetrizing the entire vector.  It is
straightforward to check and we omit the proof.  The second asserts
that $\binom{k}{\ell}(s^\ell\otimes s^{k-\ell})\iota:S^kV\to S^\ell
V\otimes S^{k-\ell}V$ approximates a sum over unshuffles
$\sS_{k,\ell}^{-1}$, and $\binom{k}{\ell}s(\iota^\ell\otimes
\iota^{k-\ell}) :S^\ell V\otimes S^{k-\ell}V\to S^kV$ approximates a
sum over shuffles $\sS_{k,\ell}$.  It is also straightforward to check
but we include the proof since it explains the combinatorial factors.
\begin{lemma}\label{lemma:symtens1}
  If $\ell\le k$, $s^k(\id^{k-\ell}\otimes (\iota s^\ell))=s^k$
\end{lemma}
\begin{definition}
  Let $\mu^{k,\ell},\nu_{k,\ell}:T^kV\to T^kV$ be given by $v\mapsto
  \sum \sigma v$, where the sum is taken over unshuffles
  $\sS_{k,\ell}^{-1}$ for $\mu$ and over shuffles $\sS_{k,\ell}$ for
  $\nu$.
\end{definition}

The maps $\mu^{k,\ell}$ and $\nu_{k,\ell}$ are defined in the tensor
algebra, but by abuse of notation, we use $\mu^{k,\ell}$ and
$\nu_{k,\ell}$ to refer to the compositions $\binom{k}{\ell}(s\otimes
s)\iota$ and $\binom{k}{l}s(\iota\otimes \iota)$ defined in the
symmetric algebra as well.
\begin{lemma}\label{L:k-l-factor}
  The following diagrams commute:
$$\xymatrix{
  S^kV\ar[d]_{\binom{k}{\ell}(s\otimes
    s)\iota}&T^kV\ar[d]^{\mu^{k,\ell}}\ar[l]_s \\S^\ell V\otimes
  S^{k-\ell}V&T^kV\ar[l]^{\quad\quad s\otimes s} } \xymatrix{ S^\ell
  V\otimes
  S^{k-\ell}V\ar[r]^{\quad\quad\iota\otimes\iota}\ar[d]_{\binom{k}{\ell}s(\iota\otimes\iota)}&T^kV\ar[d]^{\nu_{k,\ell}}
  \\S^kV\ar[r]_\iota&T^kV }$$
\end{lemma}
\begin{proof}
  Following the first diagram along the top and left gives
  \begin{align*}
    \binom{k}{\ell}(s\otimes s) \iota s(\bar{v}) &=
    \frac{1}{\ell!(k-\ell)!}(s\otimes s) \sum_{\sigma\in \sS_k} \sigma
    \bar{v}
    \\
    &=\frac{1}{\ell!(k-\ell)!}(s\otimes s) \sum_{\tau_1\in
      \sS_\ell}\sum_{\tau_2\in \sS_{k-\ell}}\sum_{\rho\in
      \sS_{k,\ell}^{-1}}(\tau_1\times \tau_2)\rho \bar{v} \\ &
    =\sum_{\rho\in
      \sS^{-1}_{k,\ell}}[\bar{v}_\ell]\otimes[\bar{v}_{k-\ell}] \\
    &=(s\otimes s)(\mu^{k,\ell} \bar{v}).
  \end{align*}
  Similarly, for the second diagram, we get
  \begin{align*}
    \binom{k}{\ell}\iota s(\iota\otimes
    \iota)([\bar{u}]\otimes[\bar{v}])&=
    \frac{1}{\ell!(k-\ell)!}\sum_{\rho\in \sS_{k,\ell}}\sum_{\tau_1\in
      \sS_\ell}\sum_{\tau_2\in
      \sS_{k-\ell}}\rho((\tau_1\iota[\bar{u}])\otimes (\tau_2\iota
    [\bar{v}]))
    \\
    &=\sum_{\rho\in \sS_{k,\ell}}\rho(\iota\otimes
    \iota)([\bar{u}]\otimes [\bar{v}]) \\
    &=\nu_{k,\ell}(\iota\otimes \iota)([\bar{u}]\otimes[\bar{v}]).
  \end{align*}
\end{proof}

Given maps between symmetric products of $V$, one can precompose with
$s$ and postcompose with $\iota$ to obtain maps between tensor
products of $V$.  The following proposition indicates the
combinatorial factor introduced when comparing the result of the
partial gluing prior to passing from symmetric to tensor (the left
hand side) and the partial gluing after passing from symmetric to
tensor (the right hand side).
\begin{proposition}\label{L:compare-circ_k}
  Let $f:S^iV\to S^jV$ and $g:S^mV\to S^nV$.  Then
$$\frac{(j+n-k)!k!}{n!j!}\iota (g\circ_k f)s
= \sum_{\substack{\sigma\in \sS^{-1}_{i+m-k,i}\\\tau\in
    \sS_{j+n-k,j-k}}}\tau((s g\iota)\circ_k (s f\iota))\sigma
$$
\end{proposition}
\begin{proof}
  The proof is a commutative diagram.  The composition along the
  righthand side of the diagram below computes
  $$\sum_{\substack{\sigma\in \sS^{-1}_{i+m-k,i}\\\tau\in
      \sS_{j+n-k,j-k}}}\tau((s g\iota)\circ_k (s f\iota))\sigma,$$ the
  righthand side of the equality in the proposition.  It will be shown
  that the composition along the top, lefthand side, and bottom of the
  diagram computes the lefthand side of the equality in the
  proposition.
$$\xymatrix{
  S^{i+m-k}V\ar[d]_{(\id\otimes f)\mu^{i+m-k,m-k}}&&T^{i+m-k}V\ar[d]^{(\id\otimes fs)\mu^{i+m-k,m-k}}\ar[ll]_s\\
  S^{m-k}V\otimes S^{j}V\ar[d]_{id\otimes ((s\otimes s) \iota)}&&T^{m-k}V\otimes S^{j}V\ar[ll]_{s\otimes \id}\ar[d]^{\id\otimes \iota}\\
  S^{m-k}V\otimes S^kV\otimes S^{j-k}V\ar[dd]_{(s (\iota\otimes \iota))\otimes \id}&T^{m-k}V\otimes S^k V\otimes S^{j-k}V\ar[l]_{s\otimes\id\otimes\id}\ar[dr]_{\id\otimes \iota\otimes\iota}&T^{j+m-k}V\ar[l]_{\id\otimes s\otimes s}\ar[d]^{\id}\\
  &S^{m-k}V\otimes T^jV\ar[d]_{\iota \otimes \id}&T^{j+m-k}V\ar[l]^{s\otimes\id}\ar[d]^{s\otimes \id}\\
  S^mV\otimes S^{j-k}V\ar[d]_{g\otimes \id}\ar@/_/[rr]_{\id\otimes\iota}&T^{j+m-k}V\ar[r]^{s\otimes \id}&S^mV\otimes T^{j-k}V\ar[d]^{(\iota g)\otimes\id}\\
  S^nV\otimes S^{j-k}V\ar[d]_{\nu_{j+n-k,n}}\ar[rr]^{\iota\otimes\iota}&&T^{j+n-k}V\ar[d]^{\nu_{n+j-k,n}}\\
  S^{n+j-k}V\ar[rr]_\iota&&T^{n+j-k}V }$$ First, we check
commutativity.  The square and triangle near the middle of the diagram
commute by Lemma \ref{lemma:symtens1}.  The rectangles at the top and
bottom commute by the construction of the shuffle and unshuffle maps.
Everything else commutes trivially.  In order to see that the
composition along the lefthand side computes
$\frac{(j+n-k)!k!}{n!j!}f\circ_k g$, consider the following
commutative diagram:
$$\xymatrix{
  S^{i+m-k}V\ar[d]_{(\id\otimes \iota f s) \iota}&&S^{i+m-k}V\ar[d]^{(\id\otimes \iota f)\mu^{i+m-k,m-k}}\ar[ll]_{\binom{i+m-k}{i}}\\
  T^{j+m-k}V\ar[rr]^{s\otimes \id}\ar[dr]^{\iota s\otimes \iota s\otimes\id}\ar[dd]_{s\otimes\id}&&S^{m-k}V\otimes T^jV\ar[d]^{\iota\otimes(\iota s)\otimes s}\\
  &T^{j+m-k}V\ar[dl]^{s\otimes \id}\ar[r]^{\id\otimes s}&T^mV\otimes S^{j-k}V\ar[d]^{s\otimes \id}\\
  S^mV\otimes T^{j-k}V\ar[rr]^{\id\otimes s}\ar[d]_{s(\iota g\otimes\id)}&&S^mV\otimes S^{j-k}V\ar[d]^{\nu_{j+n-k,n}(g\otimes \id)}\\
  S^{j+n-k}V\ar[rr]_{\binom{j+n-k}{n}}&&S^{j+n-k}V }$$ The left side
of this diagram is $g\circ_k f$ and the right hand side is the left
hand side of the previous diagram.  Here everything commutes trivially
except the triangle, which commutes due to Lemma \ref{lemma:symtens1}.
Since $$
\binom{i+m-k}{i}\binom{j+n-k}{n}=\frac{(j+n-k)!k!}{j!n!}\binom{i+m-k}{i}\frac{j!}{(j-k)!k!}$$
this computes the left hand side of the equation in the proposition,
completing the proof.
\end{proof}

\end{document}